\documentclass[12pt]{amsart}
\usepackage{amsmath,amsthm}
\usepackage{amsfonts}
\usepackage{amssymb}

\usepackage{graphicx,xy,pstricks}
\input xy
\xyoption{all}

\newcommand{\executeiffilenewer}[3]{%
 \ifnum\pdfstrcmp{\pdffilemoddate{#1}}%
 {\pdffilemoddate{#2}}>0%
 {\immediate\write18{#3}}\fi%
}

\newtheorem{lemma}{Lemma}[section]
\newtheorem{proposition}[lemma]{Proposition}
\newtheorem{theorem}[lemma]{Theorem}

\newtheorem*{conjecture}{Conjecture}

\theoremstyle{definition}
\newtheorem{definition}[lemma]{Definition}
\newtheorem{remark}[lemma]{Remark}

\newcommand{\Z}{\mathbb{Z}}
\newcommand{\R}{\mathbb{R}}
\newcommand{\Q}{\mathbb{Q}}
\newcommand{\C}{\mathbb{C}}
\newcommand{\N}{\mathbb{N}}

\renewcommand{\epsilon}{\varepsilon}

\DeclareMathOperator{\tr}{tr}
\DeclareMathOperator{\inter}{int}
\DeclareMathOperator{\vol}{vol}
\DeclareMathOperator{\covol}{covol}
\DeclareMathOperator{\card}{card}

\DeclareMathOperator{\homo}{Hom}
\DeclareMathOperator{\ev}{ev}
\DeclareMathOperator{\area}{Area}

\title{Asymptotics of quantum representations of surface groups}
\author{JULIEN MARCH\'{E} AND RAMANUJAN SANTHAROUBANE}
\date{} 

\address{Institut de Math\'ematiques, Universit\'e Pierre et Marie Curie, 75252 Paris c\'edex 05, France}
\email{julien.marche@imj-prg.fr}

\address{Institut de Math\'ematiques, Universit\'e Pierre et Marie Curie, 75252 Paris c\'edex 05, France}
\email{ramanujan.santharoubane@imj-prg.fr}
\begin{document}

\begin{abstract}
For a banded link $L$ in a surface times a circle, the Witten-Reshetikhin-Turaev invariants are topological invariants depending on a sequence of complex $2p$-th roots of unity $(A_p)_{p\in 2\N}$. We show that there exists a polynomial $P_L$ such that these normalized invariants converge to $P_L(u)$ when $A_p$ converges to $u$, for all but a finite number of $u$'s in $S^1$. This is related to the AMU conjecture which predicts that non-simple curves have infinite order under quantum representations (for big enough levels). Estimating the degree of $P_L$, we exhibit particular types of curves which satisfy this conjecture. Along the way we prove the Witten asymptotic conjecture for links in a surface times a circle. 
\end{abstract}
\maketitle
\section{Statement of the results}
\subsection{Motivation and Main result}

This paper is concerned with invariants arising from Witten-Reshetikhin-Turaev $\mathrm{SU}(2)$ topological quantum field theories (TQFT) following the skein theoretical approach of \cite{bhmv}. Such a TQFT defines for $M$ a  oriented compact $3$-manifold without boundary and $L \subset M$ a banded link, a sequence of invariants $Z_p(M,L)$ indexed by even integers $p=2r \geq 6$. For a given $p$, the invariant $Z_p(M,L)$ belongs to the cyclotomic field $K_p = \Q[A] / (\phi_{2p}(A))$,where $\phi_{2p}$ denotes the $2p$-th cyclotomic polynomial\footnote{Indeed in a finite extension of it, but we will not need it here.}. To have a numerical invariant, one needs to specify an embedding of $K_p$ into $\C$ or equivalently a $2p$-th primitive root of unity $A_p \in \C$. We will denote by $\ev_{A_p}Z_p(M,L) \in \C$ the associated numerical evaluation. An interesting question is to understand the asymptotic of the quantity $\ev_{A_p}Z_p(M,L)$ as $p \to \infty$ and as $A_p$ converges to a given number on the unit circle. When $A_p=-e^{i\pi/p}$, this problem is called the Witten asymptotic expansion conjecture. Other limits have not been studied yet with the exception of some Seifert spaces studied by Lawrence and Zagier, see \cite{LZ}. 

In this paper, we focus on the case $M= \Sigma\times S^1$ where $\Sigma$ is a compact connected oriented closed surface. We look for a formula for the quantum invariant 

\[\tr_p(L)=\frac{Z_p(\Sigma\times S^1,L)}{Z_p(\Sigma\times S^1, \emptyset)}\in K_p\]
where $L \subset \Sigma \times S^1$ is a given banded link. Notice that the quantity $Z_p(\Sigma\times S^1, \emptyset)$ is the dimension of $V_p(\Sigma)$ : the $K_p$-vector space associated to $\Sigma$ by the Witten-Reshetikhin-Turaev TQFT. Moreover, $\dim V_p(\Sigma)$ is computed by the Verlinde formula and is polynomial in $p$ with degree $0,1,3g-3$ if the genus of $\Sigma$ is $g=0,1$ or $g\ge 2$ respectively. Hence the asymptotics of the quantity $\ev_{A_p} Z_p(\Sigma \times S^1,L)$ is determined by the asymptotics of $\ev_{A_p} \tr_p(L)$. The main result of this paper is that the asymptotics of $\ev_{A_p} \tr_p(L)$ is almost determined by the evaluation of a Laurent polynomial with integral coefficients depending only on $L$.
\begin{theorem} \label{thm1}
Let $L$ be a link in $\Sigma \times S^1$. There exists a Laurent polynomial $P_L \in \Z [A^{\pm 1}]$ and a finite set $\Omega_L \subset S^1$ such that for any sequence $\{A_p \}_{p \in 2 \mathbb{N}}$ such that $A_p \underset{p \to \infty}{\longrightarrow} u \notin \Omega_L$, one has $$\ev_{A_p} \tr_p(L) = P_L(u)+O\Big(\frac{1}{p} \Big)$$
In particular the polynomial $P_L \in \Z[A^{ \pm 1}]$ is well-defined and is a topological invariant of $L$.
\end{theorem}

The polynomial $P_L$ can be viewed as a generalization of the Kauffman bracket of a link in $S^3$. Indeed if $L \subset B^3$ is a link inside a $3$-ball embedded in $\Sigma \times S^1$, the polynomial $P_L$ is nothing but the usual Kauffman bracket of $L$. For a more complicated link inside $\Sigma \times S^1$, this polynomial can be computed algorithmically, see Section \ref{skein}. The existence of such a polynomial is not clear a priori, we note that Costantino already built one in the case of a $S^2 \times S^1$ using other methods (see \cite{costantino}).

\subsection{Cyclic expansions}

In order to prove Theorem \ref{thm1}, we introduce the key notion of \emph{cyclic expansion}. We will denote by $\mathcal{C}$ the vector space of maps $f:\R\to\R$ which are continuous with compact support and piecewise polynomial. 

\begin{definition}
We will say that the sequence $\tr_p(L)$ has a cyclic expansion if there exists $P_L \in \Z[A^{\pm 1}]$, an integer $\beta \geq 0$ and a family $f_0,\ldots,f_{2\beta-1}\in \mathcal{C}$ such that 
\begin{equation}\label{cyclic}
\tr_p(L)=P_L(A)+\frac{1}{p}\sum_{\alpha=0}^{2\beta-1}  \sum_{n\in\Z } A^{2\beta n+ \alpha} f_{\alpha} \left (\frac{n}{p}\right)+O \left(\frac{1}{p}\right)
\end{equation}
\end{definition}
We need to explain the meaning of $O(\frac{1}{p})$ in Equation (\ref{cyclic}). Indeed, we have to interpret both sides as elements of $K_p\otimes \R$ endowed with the norm
$||x||_p=\inf\{\max_n |c_n|, x=\sum_{n=0}^p c_nA^n\}$.


The main technical result of this article is the following theorem. 

\begin{theorem} \label{thm existence}
For any banded link $L \subset \Sigma \times S^1$, the sequence $\tr_p(L)$ admits a cyclic expansion.

\end{theorem}

The proof of this theorem consists in a careful counting of integral points in various polytopes related to TQFT. It is postponed to Section \ref{proof cyclic}. The interest of having such a cyclic expansion is that the study of asymptotics is reduced to the following question.

Let $f$ be in $\mathcal{C}$, $\beta$ be a non zero integer and $A_p$ be a convergent sequence of $2p$-th primitive roots of unity. What is the asymptotics of $\frac{1}{p} \sum_{n\in\Z} A_p^{2\beta n} f (\frac{n}{p})$ as $p$ tends to infinity? This problem can be solved with elementary analytic tools as follows.
\begin{proposition} \label{limits}
Let $f \in \mathcal{C}$ and $\beta$ be a positive integer. Let $A_k$ be a sequence of $2p_k$-th primitive roots of unity with $p_k$ a strictly increasing sequence of even integers. 
\begin{enumerate}
\item If $\lim\limits_{k\to\infty}A_k = u$ with $u^{2\beta} \neq 1$  one has 
$$\frac{1}{p_k} \sum_{n\in\Z} A_{k}^{2\beta n} f \Big(\frac{n}{p_k}\Big) = O \Big( \frac{1}{p_k} \Big) $$

\item If $\lim\limits_{k\to\infty}A_k = u$ with $u^{2\beta} = 1$ we write $A_k=ue^{i\pi\theta_k}$ so that $\lim\limits_{k\to \infty}\theta_k=0$.
If $p_k\theta_k$ diverges when $k$ goes to infinity then 
$$\frac{1}{p_k} \sum_{n\in\Z} A_{k}^{2\beta n} f \Big(\frac{n}{p_k}\Big)=O\Big(\frac{1}{p_k\theta_k}\Big).$$

\item In the same setting as (2) suppose that $p_k\theta_k$ does not diverge. As $A_k^{2p_k}=1$, the sequence $\theta_k p_k$ takes discrete values and up to extracting a subsequence, one can suppose that it is constant equal to $ \frac{\sigma}{\beta}$, i. e. $A_k=ue^{\frac{ i\pi \sigma}{\beta p_k}}$ for some odd integer $\sigma$. Then, we have 
$$\frac{1}{p_k} \sum_{n\in\Z} A_{k}^{2\beta n} f \Big(\frac{n}{p_k}\Big) = \int_{\R} e^{2 i \pi x \sigma}f(x) dx +O\Big(\frac{1}{p_k}\Big).$$
\end{enumerate}
\end{proposition}

\begin{proof}
We set $ H_{k} = \frac{1}{p_k}\sum_{n\in\Z } A_k^{2\beta n} f(\frac{n}{p_k})$. We compute 
$$
(1-A_k^{2\beta}) H_{k} = \frac{1}{p_k}\sum_{n\in\Z} A_k^{2\beta n} \left(f\Big(\frac{n}{p_k} \Big) -   f\Big(\frac{n-1}{p_k}\Big)\right)  = O\Big(\frac{1}{p_k}\Big)$$
The last equality is obtained by applying a Taylor expansion of $f$ away from a finite number of values: this is possible since $f$ is piecewise polynomial. 

In the first case, since $u^{2\beta} \neq 1$, we deduce that $H_{k} = O(\frac{1}{p_k})$. 

In the second case, we simply observe that $1-A_k^{2\beta}\sim -2i\pi\beta\theta_k$ hence $H_k=O\big(\frac{1}{p_k\theta_k}\big)$ and we can conclude. 

In the last case, we write $H_k=\frac{1}{p_k}\sum_{n\in \Z}e^{2 i\pi\sigma\frac{n}{p_k}}f(\frac{n}{p_k})$ and recognize a Riemann sum. This gives $H_k=\int_{\R} e^{2 i \pi  \sigma x} f(x) dx +O\big(\frac{1}{p_k}\big)$. 

\end{proof}

We observe that Proposition \ref{limits} and Theorem \ref{thm existence} imply directly Theorem \ref{thm1}.

 \subsection{Applications for the AMU conjecture for surface groups}

The polynomial $P_L$ associated to the cyclic expansion can be used for proving the AMU conjecture. Let $\Sigma$ be a closed surface of genus at least $2$, $p = 2r$ be an even integer and 

\begin{equation}\label{rhop} \rho_p : \pi_1(\Sigma) \longrightarrow \prod_{n=1}^{r-1} \mathrm{PAut}(V_p(\Sigma,n))\end{equation}
 be the quantum representation\footnote{Strictly speaking, the definition given in \cite{Kob-San} was in $\mathrm{SO}(3)$-TQFT but the exact same can be done in the $\mathrm{SU}(2)$ setting.} considered by Koberda and the second author in \cite{Kob-San}. Here $V_p(\Sigma,n)$ denotes the vector space associated  by the $\mathrm{SU}(2)$ Witten-Reshetikhin-Turaev TQFT to the surface $\Sigma$ equipped with a banded point colored by $n$. It corresponds to the $\mathrm{SU}(2)$-TQFT at level $k=r-2$ in the geometric setting. Here, we use the notation of \cite{bhmv} with a shift of 1 concerning the color $n$.  A consequence of the AMU conjecture stated by Andersen, Masbaum and Ueno in \cite{amu} is the following (see \cite[Section 7]{Kob-San2} for more detail on this implication).
\begin{conjecture}[AMU conjecture for surface groups]
 If $\gamma \in \pi_1(\Sigma)\setminus\{1\}$ is not a power of a simple element then $\rho_p(\gamma)$ has infinite order for all $p$ big enough. 
\end{conjecture}
Here a non-trivial element of $\pi_1(\Sigma)$ is said simple if it is freely homotopic to a simple closed curve. We denote by $\left( \Z[A^{\pm1}]\right) ^{\times} = \{ \pm A^m \}_{ m \in \Z }$ the group of units in $\Z[A^{\pm 1}]$. 
Using cyclic expansions we define the following :

\begin{definition}\label{gammachapeau}
Let $\gamma \in \pi_1(\Sigma)$ be represented by a loop $\gamma : S^1 \to \Sigma$. We define $\hat{\gamma} \subset \Sigma \times S^1$ as the knot $ t \in S^1 \mapsto (\gamma(t),t) \in \Sigma \times S^1$ with arbitrary banded structure. For $n \in \N \setminus \{0 \} $, we denote by $(\hat{\gamma},n)$ the banded link $\hat{\gamma}$ colored by $n$ and by $\tr_p(\hat{\gamma},n)$ its normalized trace (observe that $n=1$ is the trivial color and $n=2$ the usual one). Because the indeterminacy of the banded structure, the polynomial $P_{\gamma,n}$ is well-defined up to the multiplication by a unit.
\end{definition} 

\begin{theorem} \label{thm AMU}
Let $\gamma \in \pi_1(\Sigma)$. If for some $n$,  $P_{\gamma,n}$ is neither zero nor a unit, then the AMU conjecture for surface groups holds for $\gamma$.
\end{theorem}
\begin{proof}

We have by definition $\tr_p(\hat{\gamma},n)=\frac{\tr \rho_{p,n}(\gamma)}{\dim V_p(\Sigma)}$ where $\rho_{p,n}$ is the $n$-th factor of the representation $\rho_p$ defined in Equation \ref{rhop}. If $\rho_{p,n}(\gamma)$ has finite order, then its eigenvalues are roots of unity and we have for any primitive root of unity of order $2p$ the inequality 
$$|\ev_{A_p} \tr_p(\hat{\gamma},n)|\le 1.$$
By Proposition \ref{limits}, for all but a finite number of $u\in S^1$, one has 
$$\lim\limits_{A_p\to u} \ev_{A_p} \tr_p(\hat{\gamma},n)=P_{\gamma,n}(u).$$
Hence, if $\rho_{p_k,n}$ has finite order for a sequence $p_k$ going to infinity, then $|P_{\gamma,n}(u)|\le 1$ for all $u\in S^1$. 
The theorem is then a direct consequence of Lemma \ref{parseval}.
\end{proof}

\begin{lemma}\label{parseval}
Let $P \in \Z[A^{\pm 1}]$ such that $\sup_{z \in S^1} |P(z)| \leq 1$. Then one has $P=0$ or $P \in \left( \Z[A^{\pm1}]\right) ^{\times}$.
\end{lemma}
\begin{proof}
We can write $P = \sum_{l} a_l A^l $ with the $a_j$'s in $\Z$. Let us define the continuous function $f(t)=P(e^{2 i \pi t})$. One has $$\sum_{l} |a_l |^2= \int_{0}^1 |f(t)|^2 \mathrm{dt} \le 1.  $$ Since the $a_j$'s are in $\Z$ this implies that $P=0$ or $P = \pm A^{m}$ for some $ m \in \Z$.
\end{proof}

Notice that determining if the polynomials $P_{\gamma,n}$ belong to $\left( \Z[A^{\pm1}]\right) ^{\times} \cup \{0 \} $ is similar to the problem of determining if a non-trivial knot in $S^3$  have non-trivial colored Jones polynomials, which is believed to be true. The following Theorem \ref{incompressible} will provide examples where the degree of the polynomial $P_{\gamma,3}$ can be controlled. This is reminiscent to the estimation of the colored Jones polynomials of alternating knots. 

\subsubsection{A formula for $P_L$}
We give here a formula for $P_L$ in the case where $L$ is a banded link in $\Sigma\times S^1$ whose projection on the first factor is a multicurve. 
\begin{proposition}\label{algo}
Let $p:\Sigma\times S^1\to \Sigma$ be the first projection map and consider a collection of disjoint and non-parallel annuli $T_1,\ldots, T_n\subset \Sigma$. Let $L\subset \Sigma\times S^1$ be a banded link projecting to $\bigcup_{j=1}^nT_j$ and set $L_j=L\cap p^{-1}(T_j)$. We have $P_L=\prod_{j=1}^n P_{L_j}$ so that one can reduce to the case when $L$ is inside an annulus times a circle. We put $P_L=1$ if $L$ is the empty link. 

Let $L$ be a banded link in $U=S^1\times S^1\times[0,1]$ which is a union of banded knots $\gamma_i\subset S^1\times S^1\times\{t_i\}$ for $0<t_0<\cdots<t_k<1$. Let $x_0,\ldots,x_k\in H_1(U,\Z)$ be the corresponding homology classes.  
We have $$P_L=2\sum_{\substack{\epsilon_1,\ldots,\epsilon_k=\pm 1\text{ s.t.}\\ x_0+\epsilon_1x_1+\cdots \epsilon_k x_k=0}}A^{2\area(x_0,\epsilon_1x_1,\ldots,\epsilon_kx_k)}$$
where $\area(y_0,\ldots,y_k)=0$ is the area of the polygon in $H_1(U\times S^1,\R)$ whose sides are the vectors $y_0,\ldots,y_k$. 
\end{proposition}

This formula will be useful for proving that $P_{\gamma,3}$ is non-trivial for some curves $\gamma\in \pi_1(\Sigma)$. In order to describe them, we introduce the notion of Euler incompressibility. 
\begin{definition}
A cycle in a graph $G$ is called {\it Eulerian} if it visits every edge at most once. 
An graph $G$ embedded in $\Sigma$ is said \emph{Euler-incompressible} when no Eulerian cycle of $G$ bounds a disc in $\Sigma$.
\end{definition}

\begin{theorem}\label{incompressible}
Let $\Sigma$ be a closed surface of genus $g\ge 2$ and $\gamma:S^1\to \Sigma$ be an embedding with $N$ transverse double points whose image $\Gamma=\gamma(S^1)$ is Euler-incompressible. Then, up to a unit, we have  $$P_{\gamma,3}=\sum_{i=-4N}^{4N} c_i A^i$$
where $c_{-4N}$ and $c_{4N}$ are not zero. In particular, it is not trivial if $N>0$ and Theorem \ref{thm AMU} applies.
\end{theorem}
Figure \ref{filling} shows an example of curve which fulfills the assumptions of Theorem \ref{incompressible}. Note that this loop is filling the surface in the sense that the complement of its image is a disjoint union of topological discs. According to Kra's Theorem \cite[Theorem 1.1]{Kra}, this loop is sent to a pseudo-Anosov element in the Birman Exact Sequence. To the authors' knowledge, this is the first example of a pseudo-Anosov element in the mapping class group of a genus $g\ge 2$ surface satisfying the AMU conjecture. Moreover we note that any loop whose image's complementary is a single disc fulfills the assumptions of Theorem \ref{incompressible}.
\begin{figure}[htbp]
\includegraphics[scale = 0.13]{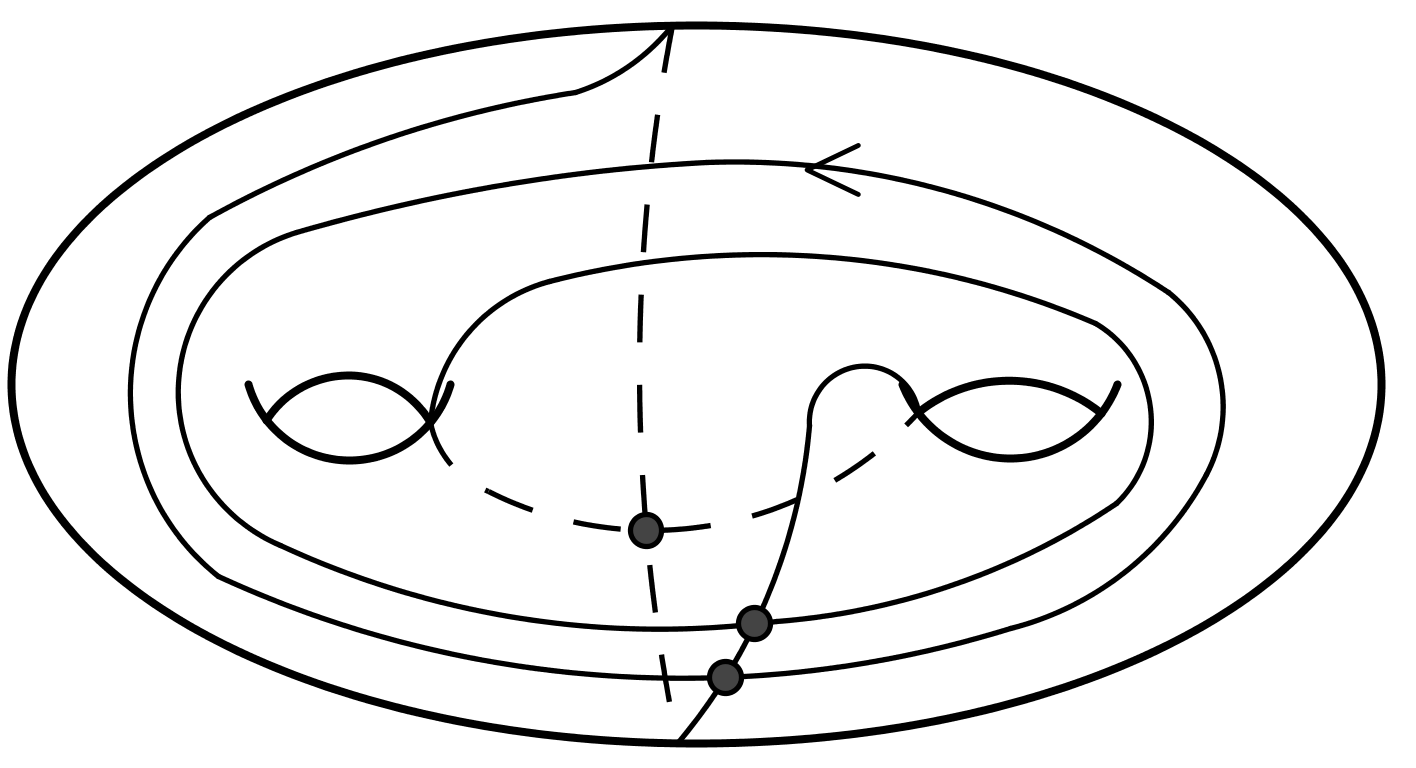}
\caption{An Euler incompressible filling curve}\label{filling}
\end{figure}

\subsection{Geometric interpretation of the cyclic expansion}

The standard root of unity $A_p=-e^{i\pi/p}$ is generally used in TQFT as the vector space $V_p(\Sigma)$ is Hermitian in that case. Most of the results or conjectures about the asymptotics of TQFT concern this setting. The Witten conjecture is the most well-known and is related to our work in the case when the underlying manifold is $\Sigma\times S^1$. 
Consider for more generality an odd integer $\sigma$ and replace $A_p$ with $A_p^\sigma$. Observe that for these roots to be primitive of order $2p$ we need that $2p$ is coprime to $\sigma$ which we assume from now on. 

Suppose that $L\subset \Sigma\times S^1$ is a banded link having a cyclic expansion as in Equation \eqref{cyclic}. Then this expansion governs the asymptotics in the sense that we have 
$$\lim_{p\to \infty} \ev_{A_p^\sigma} \tr_p(L)=P_L(-1)+\int_\R e^{2i\pi x \sigma}\left(\sum_{\alpha=0}^{2\beta-1}(-1)^\alpha f_\alpha(x)\right)dx.$$

Let $X(\Sigma\times S^1)$ be the space of conjugacy classes of irreducible representations $\rho:\pi_1(\Sigma\times S^1)\to $ SU$_2$. Denoting by $t$ the generator of $\pi_1(S^1)$, such a representation has to satisfy $\rho(t)=\pm 1$ as $t$ is central. Hence $X(\Sigma\times S^1)$ is a union of two copies of $X(\Sigma)$, defined in the same way. This manifold has dimension $6g-6$ and is endowed with a natural symplectic form $\omega$ and volume form $\nu=\frac{\omega^{3g-3}}{(3g-3)!}$. We set $\nu_g=\int_{X(\Sigma)}\nu$. 

When $\sigma=1$, the Witten conjecture predicts the following asymptotics where $L=L_1\cup\cdots\cup L_k$. 
$$\lim_{p\to \infty} \ev_{A_p} \tr_p(L)=\frac{1}{2\nu_g}\int_{X(\Sigma\times S^1)}\prod_{i=1}^k (-\tr \rho(L_i))d\nu(\rho).$$
This formula was proved in \cite{mn} in the case when $L$ lies inside $\Sigma\times [0,1]\subset \Sigma\times S^1$. This formula has also an intersection with \cite{Andersen}, where special links in finite order mapping tori were studied. Theorem \ref{witten-sigma} will cover the general case where $L$ and $\sigma$ are arbitrary. Unfortunately, the geometric meaning of the formula is less clear when $\sigma>1$. 

\begin{theorem}\label{witten-sigma}
Let $\sigma$ be an odd integer and set $A_p=-e^{i\pi/p}$. If $L\subset \Sigma\times S^1$ projects to $\Sigma$ without crossings, we have the following formula.
$$\lim\limits_{p\to \infty}\ev_{A_p^\sigma}\tr_p(L)=\frac{1}{2\nu_g}\int_{X(\Sigma\times S^1)}\prod_{i=1}^k (-\tr \rho(L_i)^\sigma)d\nu(\rho).$$
\end{theorem}
This formula extends to all links using the Kauffman relation, however we do not know a direct expression of $\lim\limits_{p\to \infty}\ev_{A_p^\sigma}\tr_p(L)$ for general banded links. The fact the $\sigma$ exponent moved from $A_p$ to $\rho(L_i)$ is a striking phenomenon which deserve further study. The theorem will be a rather direct consequence of Theorem \ref{sigma} proved in Section \ref{geom}.




\section{Skein computations}\label{skein}
\subsection{Computing the polynomial $P_L$}
The computation of the polynomial $P_L$ associated to a banded link $L \subset \Sigma \times S^1$ is better understood in terms of skein modules. For any compact oriented manifold $M$ (maybe with boundary), we denote by $\mathcal{K}(M)$ the Kauffman bracket skein module with coefficients in $\Z[A^{\pm 1}]$. We recall that it is the free $\Z[A^{\pm1}]$-module generated by the set of isotopy classes of banded links in the interior of $M$ quotiented by the following relations.  First
 $$L_{\times} = A  L_{\infty} + A^{-1}  L_0$$ where $L_{\times}$, $L_0$, $ L_{\infty}$ are any three banded links in $M$ which are the same outside a small $3$-ball but differ inside as in Figure \ref{Ktriple}. In this case, the triple $L_{\times}$, $L_0$, $ L_{\infty}$ is called a Kauffman triple. The second relation satisfied in $\mathcal{K}(M)$ is $L \cup D=-(A^2+A^{-2}) \, L$ where $L$ is any link in $M$ and $D$ is a trivial banded knot.

\begin{figure}[htbp]
\includegraphics[width=3.2cm]{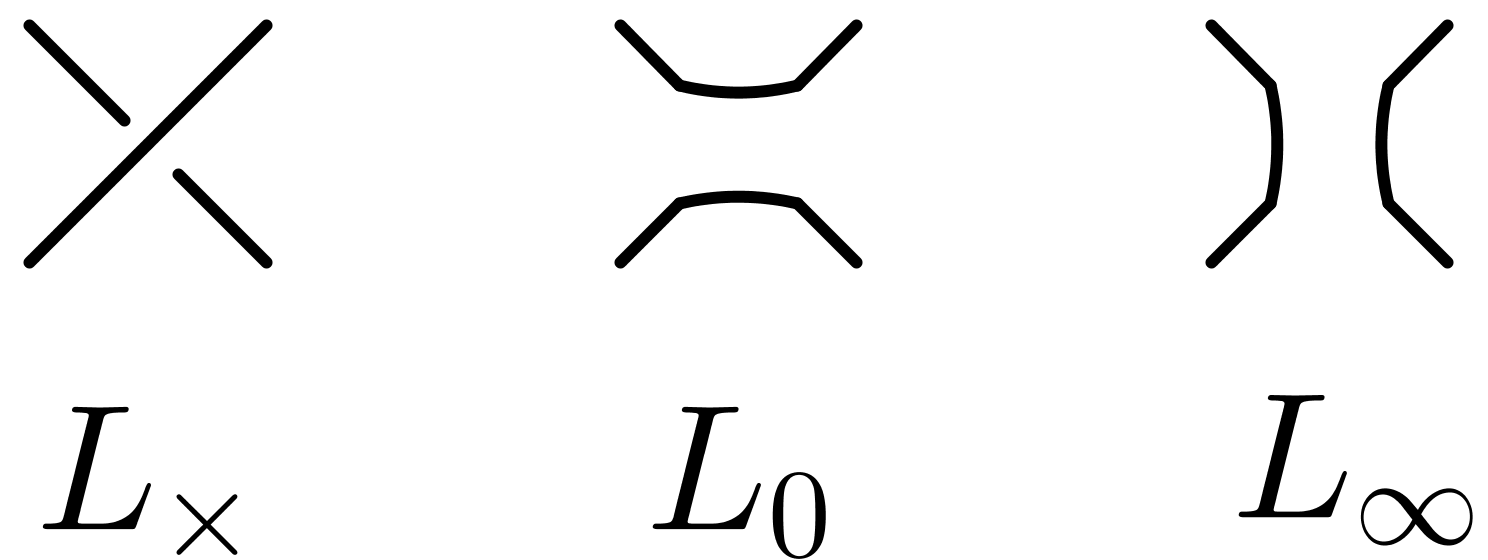}
\caption{Kauffman triple}\label{Ktriple}
\end{figure}
\begin{proposition} \label{skein_cyclo}
Let $L_{\times}$, $L_0$, $ L_{\infty}$ be a Kauffman triple in $\Sigma \times S^1$. If the sequences $\tr_p(L_0)$ and $\tr_p(L_{\infty})$ have a cyclic expansion then the sequence $\tr_p(L_{\infty})$ has a cyclic expansion. Moreover $$P_{L_{\times}}= A P_{L_ \infty} + A^{-1} P_{L_0}$$

\end{proposition}

\begin{proof}
Remark that the Witten-Reshetikhin-Turaev invariants satisfy the skein relation : $\tr_p(L_{\times}) = A \tr_p(L_{\infty})+ A^{-1} \tr_p(L_0)$. Therefore, it is enough to prove that cyclic expansions well behave under multiplication by $A^{ \pm 1}$ and finite sum. For $\beta \ge 0$ an integer, $f=(f_0,...,f_{2\beta-1}) \in \mathcal{C}^{2 \beta}$ we define $$ H_p(\beta,f) = \frac{1}{p} \sum_{\alpha =0}^{2 \beta-1} \sum_{n \in \Z} A^{\alpha+ 2n \beta}f_{\alpha}\Big( \frac{n}{p} \Big)$$

Let $\beta > 0$ be an integer and let $f=(f_0,\ldots, f_{2\beta-1}) \in \mathcal{C}^{2 \beta}$. Applying a Taylor expansion to the functions $f_0,\ldots,f_{2\beta-1}$ we have the following $(1-A^{2 \beta})H_p(\beta,f) = O (\frac{1}{p})$. Hence the  sequence $A^{\pm 1} H_p(\beta,f)$ admits a cyclic expansion. This says that if a sequence admits a cyclic expansion then the same sequence multiplied by $A^{\pm 1}$ also admits a cyclic expansion.

Let $\beta' > 0$ be an integer divisible by $\beta$ and set $\delta = \beta'/ \beta$. One can make the Euclidean division $n=\delta m+\alpha'$ and write
$$H_p(\beta,f) = \frac{1}{p} \sum_{\alpha =0}^{2 \beta-1} \sum_{\alpha' = 0}^{\delta-1} \sum_{m \in \Z} A^{\alpha+ 2\beta\alpha'+2 m \beta'}f_{\alpha}\Big( \frac{\delta  m + \alpha' }{p} \Big) $$
By the previous argument and the estimation $f_{\alpha}( \frac{\delta  m + \alpha' }{p} )=f_\alpha(\frac{\delta m}{p})+O(\frac{1}{p})$, we find that there exists $g \in \mathcal{C}^{2 \beta'}$ such that $H_p(f,\beta) = H_p(g,\beta')+O(\frac{1}{p})$. From this we can deduce that the sum of two sequences admitting cyclic expansions also admits a cyclic expansion.
\end{proof}
This proposition means that the polynomial $P_L$ extends to a $\Z[A^{\pm 1}]$-linear map $\eta: \mathcal{K}(\Sigma\times S^1)\to \Z[A^{\pm 1}]$ defined by $\eta(L)=P_L$. If $L$ sits inside a ball $B\subset \Sigma\times S^1$, the polynomial $P_L$ is just the Kauffman bracket of $L$. Hence the map $\eta$ defines a section of the natural inclusion map $\mathcal{K}(B)\to \mathcal{K}(\Sigma\times S^1)$. We will construct $\eta$ by giving its value on a $\Z[A^{\pm 1}]$-span of $\mathcal{K}(\Sigma \times S^1)$. That this map is well-defined is a non-trivial consequence of the existence of TQFT invariants and properties of the cyclic expansions.

Before that, we need to recall the results concerning the multiplicative structure of the skein module of the torus times an interval.
\subsubsection{Review of the skein module of the torus times an interval} \label{frogel0}
 We denote by $T$ the torus $S^1 \times S^1$.
It was proven by Frohman-Gelca and Sallenave (see \cite{fro-gel,sallenave}) that the skein module of the torus $T$ is isomorphic to the symmetric part of the quantum torus. Formally, we define the quantum torus as the non-commutative $\Z[A^{\pm 1}]$-algebra $\mathcal{T}=\Z[A^{\pm 1}]\langle M^{\pm 1},L^{\pm 1}\rangle /(LM-A^2ML)$. Let $\sigma$ be the involution of $\mathcal{T}$ defined by $\sigma(M^mL^l)=M^{-m}L^{-l}$. 

\begin{proposition}
For any $(m,l)\in \Z^2$ we set 
$$(m,l)_T=(-1)^{l+m}A^{ml}(M^mL^l+M^{-m}L^{-l})\in \mathcal{T}^{\sigma}.$$
There is an isomorphism of algebras $\Upsilon:\mathcal{K}(T\times [0,1])\overset{\sim}{\to} \mathcal{T}^{\sigma}$
which maps the standard banded curve in $T\times[0,1]$ with slope $(m,l)$ to the element $(m,l)_T$ when $\gcd(m,l)=1$. 

We have for any $a,b,c,d\in \Z$ the following product-to-sum formula:
$$(a,b)_T(c,d)_T=A^{ad-bc}(a+c,b+d)_T+A^{-ad+bd}(a-c,b-d)_T.$$
\end{proposition}

For technical reasons, we will also need the following normalization $\langle l,m\rangle_T =M^mL^l+M^{-m}L^{-l}$.
\begin{remark} \label{remfrogel}
This proposition implies that $\mathcal{K}(T\times [0,1])$ is generated as a $\Z[A^{\pm1}]$-module by $\{ \langle m,l \rangle_T  \, | \, (m,l) \in \Z^2 \}\cup \{ \emptyset \}$.
\end{remark}

\subsubsection{Weighted multicurves} We introduce weighted multicurves as a set of generators of $\mathcal{K}(\Sigma\times S^1)$ over $\Z[A^{\pm 1}]$.  
\begin{definition}
Let $k \geq 0$, a $k$-multicurve is a collection of $k$ disjoint essential simple oriented and non pairwise parallel closed curves.
\end{definition}
\begin{definition}
Let $\gamma = \gamma_1 \cup \cdots \cup \gamma_k$ be a $k$-multicurve. A weight on $\gamma$ is an element $w=(a_1,b_1,\ldots,a_k,b_k)\in  \Z^{2k}$ thought as an assignment of a pair $(a_i,b_i)$ for each connected component $\gamma_i$ of $\gamma$. A pair $(\gamma,w)$ will be called a weighted multicurve.
\end{definition}

Let $\gamma = \gamma_1 \cup \cdots \cup \gamma_k$ be a $k$-multicurve on $\Sigma$ with weight $w$. 
For $1 \leq j \leq k$, we choose a diffeomorphism between $S^1$ and $\gamma_j$ respecting the orientation. Denote by $T$ the torus $S^1\times S^1$ :  we can embed $T$ in $\Sigma\times S^1$ by sending the first factor to $\gamma_j$. This embedding extends to an embedding $\Phi_j:T\times [0,1]\to \Sigma\times S^1$ respecting the orientation of $\Sigma\times S^1$. 

\begin{definition}
The skein associated to the weighted multicurve $(\gamma,w)$ is the element $[\gamma,w]=\bigcup_{j=1}^k \Phi_j(\langle a_j,b_j\rangle_T)\in \mathcal{K}(\Sigma\times S^1)$ where $w=(a_1,b_1,\ldots,a_k,b_k)$ and $\langle a,b\rangle_T$ is the skein element defined in Subsection \ref{frogel0}.
\end{definition}

\begin{proposition} \label{multi_span}
The set of weighted multicurves $[\gamma,w]$ spans the $\Z[A^{\pm 1}]$-module $\mathcal{K}(\Sigma\times S^1)$.
\end{proposition}

\begin{proof}
Let $L \subset \Sigma \times S^1$ be a banded link. Set $J = \{ e^{i \pi t} \, | \, t \in [0,1] \}$ and $J^*=S^1 \setminus J$. There is a finite set $\{p_1,... p_n \}$ of banded points in $\Sigma$ (perhaps empty) so that one has up to isotopy
 $L \cap (\Sigma \times J^*) = \{p_1,...,p_n \} \times J^*$.
 Let $L'$ be the intersection of $L$ with $\Sigma \times J$. We have the following picture :

\begin{figure}[htbp]
\centering
  \def\svgwidth{6cm}
 \executeiffilenewer{multi.svg}{multi.pdf}%
 {inkscape -z -D --file=multi.svg %
 --export-pdf=multi.pdf --export-latex}%
 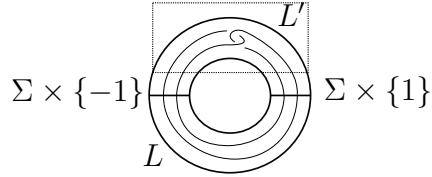%

  \caption{A banded link in $\Sigma\times S^1$}
  \label{idempotent} 
\end{figure}

Applying the skein relations we can reduce to the case where the projection of $L'$ on $\Sigma$ is finite disjoint union of banded simple closed curves and banded points. Hence we can find an integer $k \geq 0$ and a $k$-multicurve $ \gamma = \gamma_1 \cup ... \cup \gamma_k$ in $\Sigma$ such that, up to isotopy, $L \subset \tilde{\gamma} \times S^1$ where $\tilde{\gamma} \subset \Sigma$ is a tubular neighborhood of $\gamma$. This says that $L$ is in the image of the canonical map $\mathcal{K}(T \times [0,1])^{\otimes k}\to \mathcal{K}(\Sigma \times S^1)$ induced by the inclusion $\tilde{\gamma} \times S^1 \hookrightarrow \Sigma \times S^1$. Finally from Remark \ref{remfrogel} we conclude that $L$ is a $\Z[A^{\pm1}]$-linear combination of weighted multicurves where the underlying multicurve is $\gamma$.

\end{proof}
Theorem \ref{thm1} will follow from the following one which will be proved in Section \ref{proof cyclic}.
\begin{theorem}\label{existence}
For any weighted curve $[\gamma,w]$ the sequence $\tr_p\left([\gamma,w]\right)$ has a cyclic expansion. Let $k$ be the number of connected components of $\gamma$. One has $\eta([\gamma,w])=2^k$ if $w= 0$ and $0$ otherwise.
\end{theorem}

This theorem conjugated with the following lemma proves Proposition \ref{algo}. 
\begin{lemma}\label{area}
Let $T$ be the standard torus and $\alpha$ be the 1-form on $H_1(T,\R)$ given by $\alpha_x(y)=\det(x,y)$ where $\det$ stands for the intersection product. 

Let $x_0,\ldots,x_k$ be vectors in $H_1(T,\Z)$. We denote by $P(x_0,\ldots,x_k)$ the concatenation of the segments generated by the vectors $x_0,\ldots,x_k$. For $x\in H_1(T,\Z)$, we denote by $(x)_T$ the corresponding vector in $\mathcal{K}(T\times [0,1])$. We have then the following formula:

$$(x_0)_T\cdots (x_k)_T=\sum_{\epsilon_1,\ldots,\epsilon_k=\pm 1} A^{\int_{P(x_0,\epsilon_1 x_1,\ldots,\epsilon_k x_k)}\alpha}(x_0+\epsilon_1x_1+\cdots \epsilon_kx_k)_T$$
\end{lemma}
\begin{proof}
This is a generalization of the product-to-sum formula which can be proved by an immediate induction, observing that $\int_{P(x_0,\ldots,x_{k+1})}\alpha=\int_{P(x_0,\ldots,x_k)}\alpha+\det(x_0+\cdots+x_k,x_{k+1})$. 
\end{proof}
The main application of this lemma will concern the case when $x_0+\epsilon_1x_1+\cdots \epsilon_kx_k=0$ in which case Stokes formula and the equality $d\alpha=2\det$ implies that $\int_{P(x_0,\epsilon_1x_1,\ldots,\epsilon_kx_k)}\alpha=2\area(x_0,\epsilon_1x_1,\ldots,\epsilon_k x_k)$. 

\subsection{Application to the AMU conjecture}

In this section, we prove Theorem \ref{incompressible}.   

We define the degree of a non-zero Laurent polynomial by the following formula:
$$\deg P=\inf \{n\in \N, P(A)=\sum_{i=-n}^n c_i A^i\}.$$

We fix a surface $\Sigma$ of genus $g$ and start the proof with two technical lemmas. 
\begin{definition} \label{graph height}
Let $G$ be a graph embedded in $\Sigma$ with one oriented edge. 
Consider a tubular neighborhood $V$ of $G$ and a map $f:V\to S^1$ which is constant equal to $1$ out of the labeled edge and makes one positive turn along the oriented edge. We define the following banded graph:
$$\hat{G}=\{(x,f(x)),x\in V\}\subset \Sigma\times S^1.$$ Moreover we denote by $\partial \hat{G}$ the banded link in $\Sigma \times S^1$ defined by the boundary of $\hat{G}$.
\end{definition}

Given a banded graph $G\subset \Sigma\times S^1$. We say that a component $C$ of $G$ bounds a disc if there is an embedded disc $D\subset \Sigma\times S^1$ such that $D\cap G=C$. 

\begin{lemma} \label{bound}
Let $G$ be a graph embedded in $\Sigma$ with one oriented edge and suppose that the component of $G$ containing the arrow is not the boundary of a disc. Then the degree of $\eta (\partial \hat{G})$ is bounded by twice the number of components of $\partial \hat{G}$ bounding a disc in $\Sigma\times S^1$. 
\end{lemma}

\begin{proof}
Each trivial component produces a factor $-[2]$ by the Kauffman relations: this shows that the bound is optimal. Remove all these components from $\partial \hat{G}$. We are reduced to prove that the polynomial $\eta (\partial \hat{G})$ has degree zero. Represent the directed edge by 2 arrows on the multicurve $\partial G\subset \Sigma$. The lemma follows from a case by case study of the possible configurations of arrows.

{\bf Case 1:} The two arrows belong to the same component and cancel. This component is non-trivial as we removed them before starting. We conclude with the following observation: the polynomial associated to a multicurve in $\Sigma$ is a constant. Indeed, the multicurve is a union of parallel copies of curves of type $(1,0)_T$. Lemma \ref{area} gives $\eta (1,0)_T^n=\binom{n}{n/2}$ which is an integer.

{\bf Case 2:} The two arrows belong to the same component $\gamma$ and add. This is indeed impossible by considering $\gamma$ as a boundary curve of the component of the banded graph containing the arrows. By construction, the two arrows are in opposite directions relatively to this orientation and hence cannot add. 



{\bf Case 3:} The two arrows belong to two non parallel components. Hence, one of them is non-trivial and its neighborhood has the form $x=(1,0)_T^n(1,1)_T(1,0)_T^m$. Invoking Lemma \ref{area}, there are no closed path in the expansion of $x$, hence $\eta(x)=0$. 

{\bf Case 4:} The two arrows belong to two parallel and trivial components. Then these curves bound a trivial circle containing the arrow which is forbidden by assumption. 

{\bf Case 5:} The two arrows belong to parallel and consecutive non-trivial components. Then this corresponds to $x=(0,1)_T^n (1,1)_T^2(0,1)_T^m$. Again by Lemma \ref{area}, any closed path in the  expansion of $x$ has vanishing area, hence $\eta(x)$ is an integer. 

{\bf Case 6:} The two arrows belong to parallel and non-consecutive non-trivial components. Then, noticing that there exists an embedded arc which join them, we are in the configuration $(1,1)_T(1,0)_T^n(1,1)_T$. We conclude as in the Case 5. 
\end{proof}

\begin{lemma} \label{inequalities}
Let $G$ be an embedded graph with one oriented edge. Let $n$ and $e$ be respectively the number of vertices and edges of $G$. Let 
\begin{itemize}
\item[-]$u$ be the number of components of $\partial\hat{G} $ bounding simply connected components of $\hat{G}$
\item[-]$v$ be the number of components of $\partial\hat{G} $ bounding a disc and not counted in $u$. 
\end{itemize}
Suppose that $G$ is Euler-incompressible and at most quadrivalent. Then we have the following inequalities $$e+u\le 2n\quad \text{and} \quad  v \le n$$ 
Moreover $v-n = u+e-2n = 0$ implies that $G$ is a disjoint union of circles (none of them bound a disc in $\Sigma$ since $G$ is Euler-incompressible).
\end{lemma}
\begin{proof}
Let $(G_i)_{i \in I}$ be the connected components of $G$ with negative Euler characteristic. 

Let $n_i$ (resp. $e_i$) be the number of vertices (resp. edges) of $G_i$. As $G_i$ is not a circle, one has $n_i>0$. Since $G_i$ is at most quadrivalent, we have $e_i\le 2n_i$ and hence $ -\chi(G_i) \le n_i$. We compute $$-\chi(G) = e-n = -u+ \sum_{i \in I} -\chi(G_i) \le -u + \sum_{i \in I} n_i \le -u+n$$ Form which we conclude $e+u\le2n$. 

Now let $w_i$  be the number of boundary components of $G_i$ and  $v_i$ be the number of boundary components of $G_i$ bounding a disc in $\Sigma$. Since $G_i$ is Euler-incompressible, we have $ \sum_{i \in I} v_i = v$. 

Consider the closed surface $S_i$ obtained by gluing discs to the boundary components of $G_i$. If $S_i$ is a sphere then any boundary component of $G_i$ is an Eulerian cycle which implies $v_i=0$ by Euler-incompressibility. The equation $v_i\le n_i$ follows in that case. If $S_i$ is not a sphere, we have $\chi(S_i)\le 0$.  From $e_i \le 2n_i$ and $v_i \le w_i$ we get \begin{equation} \label{eq0} -n_i+v_i \le n_i-e_i+v_i \le \chi(S_i)= n_i-e_i+w_i \le 0 \end{equation}  In any cases, $v_i\le n_i$, and summing over $i\in I$ we get $v \le  n$.

It remains to prove the last part of the lemma. Suppose that $v-n=u+e-2n =0$ and fix $i \in I$.  We have $v_i = n_i$. If $S_i$ is a sphere we have $v_i=n_i=0$ : this is not possible since a graph should have at least one vertex. Therefore (\ref{eq0}) implies $\chi(S_i)=0$ and $w_i=v_i$. Hence $G_i$ is embedded in a torus $S_i$ in such a way that any of its boundary components bounds a disc in $\Sigma$. This allows to define an embedding $S_i\to \Sigma$ : since the genus of $\Sigma$ is at least two, this is not possible. We conclude that $I$ is empty and $G$ is a disjoint union of circles.

\end{proof}

\begin{proof}[Proof of Theorem \ref{incompressible}]
We denote by $[\hat{\gamma},3]\in \mathcal{K}(\Sigma\times S^1)$ the skein element obtained by coloring $\hat{\gamma}$ (see Definition \ref{gammachapeau}) by the color 3 (2 in the \cite{bhmv} setting). In the setting of \cite{bhmv}, this corresponds to the insertion of the idempotent $f_2$ in $\partial\hat{\gamma}$ as shown in Figure \ref{idempotent}. 

\begin{figure}[htbp]
\centering
  \def\svgwidth{6cm}
 \executeiffilenewer{f2.svg}{f2.pdf}%
 {inkscape -z -D --file=f2.svg %
 --export-pdf=f2.pdf --export-latex}%
 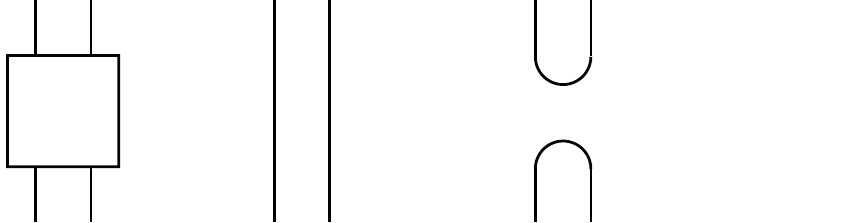%

  \caption{The idempotent $f_2$}
  \label{idempotent} 
\end{figure}
For any banded graph $\Gamma\subset\Sigma \times S^1$, we will denote by $[\Gamma,3]$ the element of $\mathcal{K}(\Sigma\times S^1)\otimes \Q(A)$ obtained by inserting the idempotent $f_2$ in $\partial \Gamma$ at all edges of $\Gamma$. 
Setting $[2]=A^2+A^{-2}$, these skein elements satisfy the following skein relation: 


\begin{equation} \label{skein_eq}  \begin{minipage}[c]{0.8cm}
\includegraphics[scale = 0.1]{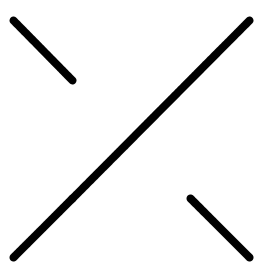}
\end{minipage} = A^4 \, \begin{minipage}[c]{0.8cm}
\includegraphics[scale = 0.1]{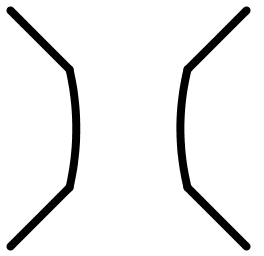}
\end{minipage}+ [2] \, \, \begin{minipage}[c]{0.8cm}
\includegraphics[scale = 0.1]{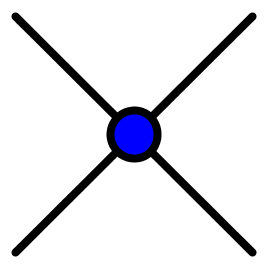}
\end{minipage} + A^{-4} \, \, \begin{minipage}[c]{1cm}
\includegraphics[scale = 0.1]{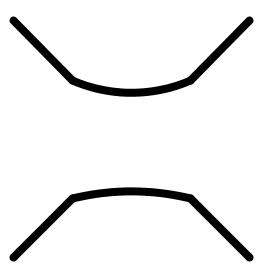}
\end{minipage} \end{equation}

We extend the map $\eta:\mathcal{K}(\Sigma\times S^1)\to \Z[A^{\pm 1}]$ by tensoring with $\Q(A)$ so that we can evaluate banded colored graphs. The proof will consist in investigating all terms in the state sum suggested by Equation \eqref{skein_eq}. 

Set $\Gamma=\gamma(S^1)\subset \Sigma$ and orient the edge going through the base point. 
Let $V$ be the set of vertices of $\Gamma$. Given $S : V \to \{-1,0,1 \}$, we define $\Gamma_S$  to be the graph obtained by transforming all vertices of $\Gamma$ as follows : if a vertex evaluates to $1$ under $S$ it is replaced by a positive smoothing, if a vertex evaluates to $-1$ under $S$ it is replaced by a negative smoothing and if a vertex evaluates to $0$ under $S$ it is not changed. $\Gamma_S$ can be viewed as a quadrivalent graph in $\Sigma$ with one directed edge. Consider the banded graph $\hat{\Gamma}_S$ in $\Sigma \times S^1$ defined by Definition \ref{graph height}. 

For $S : V \to \{-1,0,1 \}$, we denote by $a_S,b_S,n_S$ respectively the number of preimages of $1,-1,0$. Applying the rule \eqref{skein_eq} to all the crossings of $\gamma$ gives 
$$ \eta([\hat{\gamma},3]) = \sum_{S : V \to \{-1,0,1\}} \, A^{4(a_S-b_S)}[2]^{n_S} \eta[\hat{\Gamma}_S,3] $$ 

Consider one term of the sum associated to a map $S$. To avoid heavy notations, we remove the subscript $S$ to $a,b$ and $n$. We have the decomposition $N=n+a+b$ and $n$ is the number of vertices of $\Gamma_S$. The original graph $\Gamma$ was a quadrivalent graph with $N$ vertices hence satisfying $\chi(\Gamma)=-N$. The graph $\Gamma_S$ being obtained by smoothing $a+b$ vertices satisfies $\chi(\Gamma_S)=-N+a+b=-n$. 

Computing $\eta [\hat{\Gamma}_S,3]$ involves inserting idempotents in all edges of $\Gamma_S$. Let $E_S$ be the set of edges of $\Gamma_S$. For $\xi : E_S \to \{ 0 ,1 \}$, we define $\Gamma_{S,\xi}$ to be the graph obtained by deleting all edges $e$ such that $\xi(e) = 1$. We define $s_\xi = \card(\xi^{-1}(\{1 \})$. According to the rule of Figure \ref{idempotent}, we have $$ \eta[\hat{\Gamma}_S,3] = \sum_{\xi : E_S \to \{0,1 \}} [2]^{-s_\xi}\eta(\partial \hat{\Gamma}_{S,\xi})$$
Now fixing $\xi : E_S \to \{ 0 ,1 \}$ we want to bound $ \deg A^{4a-4b}[2]^{n-s}\eta( \partial \hat{\Gamma}_{S,\xi})$ (here $s$ stands for $s_\xi$). Note that $\chi(\Gamma_{S,\xi})=-n+s$ and recall that as $\Gamma$ is Euler-incompressible, this property still holds for $\Gamma_{S,\xi}$. Hence Lemma \ref{bound} applies for $\Gamma_{S,\xi}$ and we have
\begin{equation}\label{eq1} \deg A^{4a-4b}[2]^{n-s}\eta( \partial \hat{\Gamma}_{S,\xi}) \le 4|a-b|+2(n-s+c)\end{equation} where $c$ is the number of components of $\partial \hat{\Gamma}_{S,\xi}$ which bound a disc in $\Sigma$. Let us write $c = u+v$ where $u$ is the number of simply connected components of $\Gamma_{S,\xi}$. 
 
We denote by $T$ the left-hand side of Equation \eqref{eq1}. Supposing that $a\ge b$, we can apply Lemma \ref{inequalities} to $\Gamma_{S,\xi}$. Hence (\ref{eq1}) gives $T \le 4N-8b+2(v-n)+ 2(u-s) \le 4N$ (we recall that $a+b+n=N$ and $\chi(\Gamma_{S,\xi})=-n+s=n-e$). This is strictly less that $4N$ unless $b=0$ and $n=0$ (by Lemma \ref{inequalities}), that is unless we smoothed all the vertices of $\Gamma$ in the positive way. Doing so, $\Gamma_S$ is a  collection of non-trivial curves (since $\Gamma$ is Euler-incompressible) colored by $3$, one of them being oriented.  

The computation can be done locally in an annulus times an interval. Expanding the idempotent $f_2$, an unoriented curve colored by $3$ is $[(1,0),3]=(1,0)_T^2-[\emptyset]=(2,0)_T+[\emptyset]$ and similarly, an unoriented curve colored by $3$ is equal to $[(1,1),3]=(2,2)_T+[\emptyset]$. Hence expanding $\eta([(1,0),3]^n[(1,1),3][(1,0),3]^m)$ with the help of Lemma \ref{area} gives a positive result. 
Repeating the argument for negative smoothing, we finally proved Theorem \ref{incompressible}. 
\end{proof}

\section{Proof of the cyclic expansion} \label{proof cyclic}
\subsection{Curve operators acting on the torus}\label{frogel}

Let $(e_l)_{l=1,\ldots,r-1}$ be the basis of $V_p(T)$ obtained by filling $T$ with $D^2\times S^1$ and coloring the core curve by $l$. Notice that we use the convention of \cite{lj}, that is $e_l$ corresponds to $(-1)^{l-1}u_{l-1}$ where $u_l$ is the basis used in \cite{bhmv}. 
The algebra $\mathcal{K}(T\times [0,1])$ acts naturally on $V_p(T)$ by stacking. Let us cover this action by an action of the quantum torus. 

Define $E_p=\bigoplus_{l\in \Z/p\Z}K_p(\sqrt{2})\theta_l$ as a $p$-dimensional Hermitian vector space formally generated by $\theta_l$ where $\langle \theta_l,\theta_m\rangle=\delta^p_{l-m}$. By convention, we set $\delta^p_i=1$ if $p|i$ and $0$ otherwise. The quantum torus $\mathcal{T}$ acts on $E_p$ by the formulas $M\theta_l=A^{2l}\theta_l$ and $L\theta_l=\theta_{l+1}$. 

\begin{lemma}
Let $\sigma$ be the involution of $E_p$ defined by $\sigma (\theta_l)=\theta_{-l}$. Then the map $V_p(T)\to E_p$ defined by $e_l\mapsto \frac{1}{\sqrt{2}}(\theta_l-\theta_{-l})$ is an isometry onto the $\sigma$-antisymmetric part which commutes with the action of $\mathcal{T}^\sigma$.
\end{lemma}
\begin{proof}
This is well-known, see for instance \cite[Section 2]{lj}. Indeed, we get the expected formulas 
$$(1,0)_Te_l=(-M-M^{-1})e_l=(-A^{2l}-A^{-2l})e_l$$ $$(0,1)_Te_l=(-L-L^{-1})e_l=-e_{l+1}-e_{l-1}.$$ 
\end{proof}

\begin{lemma}\label{calcul}We have the formula 
$$\langle \langle a,b\rangle_T e_l,e_m\rangle=A^{2a(l+b)}(\delta^p_{l+b-m}-\delta^p_{l+b+m})+A^{2a(-l+b)}(\delta^p_{l-b-m}-\delta^p_{l-b+m}).$$
\end{lemma}
\begin{proof}
This is a direct computation, replacing $e_l$ with $\frac{1}{\sqrt{2}}(\theta_{l}-\theta_{-l})$. 
\end{proof}

\subsection{Reduction to weighted multicurves} \label{multi}
We remark that because Proposition \ref{multi_span} and Proposition \ref{skein_cyclo}, it is enough to prove that any weighted multicurve has cyclic expansion.

Let $\gamma=\gamma_1\cup\cdots \cup \gamma_k$ be a $k$-multicurve. We denote by $\Sigma_\gamma$ the surface obtained by surgery on every component of $\gamma$. For any $i$, we add two marked points $p_i,q_i$ from each side of the handle used in the surgery at $\gamma_i$ as in Figure \ref{fig_surgery}. 
 \begin{figure}[htbp]
\centering
  \def\svgwidth{8cm}
 \executeiffilenewer{surgery.svg}{surgery.pdf}%
 {inkscape -z -D --file=surgery.svg %
 --export-pdf=surgery.pdf --export-latex}%
 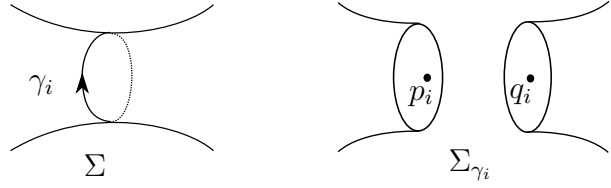%

  \caption{Surgery on $\Sigma$}
  \label{fig_surgery} 
\end{figure}

Finally, given a vector $x=(n_1,m_1,\ldots,n_k,m_k)\in \{1,\ldots,r-1\}^{2k}$ thought as a coloring of the points $p_i,q_i$, we set 
$$ C_{\gamma,p}(x) = \dim V_p(\Sigma_{\gamma},n_1,m_1,\ldots,n_k,m_k).$$

\begin{proof}[Proof of Theorem \ref{existence}] Because of 
Let $\gamma = \gamma_1 \cup ... \cup \gamma_k$ be a multicurve with weight $w=(a_1,b_1,\ldots,a_k,b_k)$.
Let $\gamma_j'$ and $\gamma_j''$ be two parallel copies of $\gamma_j$ such that $\Phi_j(T\times\{0,1\})=\gamma_j'\times S^1\cup \gamma_j''\times S^1$. 
Consider the basis $(e_i)_{i=1,\ldots,r-1}$ of $V_p(\gamma_j'\times S^1)$ obtained by filling $\gamma_j'\times S^1$ with $D^2\times S^1$ (the same with $\gamma_j''$).

The TQFT axioms imply the following equation where $n$ and $m$ are elements of $\{1,\ldots,r-1\}^k$. 

$$
\tr_p [\gamma,w]  = \sum_{n,m}  \left( \prod_{j=0}^k \langle\langle a_j,b_j\rangle_T \, e_{n_j},e_{m_j}\rangle\right) \, \frac{{C}_{\gamma,p}(n,m)}{ Z_p(\Sigma \times S^1)}$$



Using Lemma \ref{calcul}, we have
\begin{multline}
\tr_p [\gamma,w] =\sum_{n,m}\prod_{j=1}^k\Big( A^{2n_ja_j}(\delta^{p}_{n_j+b_j-m_j}-\delta^{p}_{n_j+b_j+m_j})+\\
A^{-2n_ja_j}(\delta^{p}_{n_j-b_j-m_j}-\delta^{p}_{n_j-b_j+m_j}) \Big)\frac{{C}_{\gamma,p}(n,m)}{ Z_p(\Sigma \times S^1)}
\end{multline}

Introducing signs $\epsilon$ and $\zeta$ in $\{\pm 1\}^k$, we can rewrite the sum in the following way
$$\tr_p [\gamma,w]  = \sum_{n,m,\epsilon,\zeta} \prod_j A^{2\epsilon_jn_ja_j}\zeta_j\delta^{p}_{n_j+\epsilon_jb_j-\zeta_jm_j} \frac{{C}_{\gamma,p}(n,m)}{ Z_p(\Sigma \times S^1)} $$
We conclude using Proposition \ref{comptage} where $x$ stands for the $2k$-tuple $(n_1,m_1,\ldots,n_k,m_k)$, $l_j(x)=n_j-\zeta_j m_j$, $\alpha_j=-\epsilon_jb_j$ for $j=1,\ldots,k$, $D=\gcd(a_1,\ldots,a_k)$ and $l_0(x)=\frac{1}{D}\sum_j n_ja_j$. 
\end{proof}

\subsection{Counting points in polytopes}
\begin{proposition}\label{comptage}
Let $D$ be an integer, $L=(l_0,\ldots,l_k)$ be a surjective linear map $L:\Z^{2k}\to \Z^{k+1}$, and $\alpha=(\alpha_1,\ldots,\alpha_k)\in \Z^k$ some integers. Then setting
\begin{equation}\label{formuleF}
F_p(l,\alpha)=\sum_{x\in \{1,\cdots,r-1\}^{2k}}A^{2Dl_0(x)}\prod_{j=1}^k\delta^p_{l_j(x)-\alpha_j}C_{\gamma,p}(x)
\end{equation}
the quantity $\frac{F_p(l,\alpha)}{Z_p(\Sigma\times S^1)}$ has a cyclic expansion. 
\end{proposition}
\begin{proof}
Let $H$ be a disjoint union of handlebodies bounding $\Sigma_\gamma$ and $\Gamma$ be a unitrivalent banded graph embedded in $H$ such that $\Gamma\cap \Sigma=\partial \Gamma=\{p_1,q_1,\ldots,p_k,q_k\}$ and such that $H$ retracts on $\Gamma$. We will denote by $E$ the set of edges of $\Gamma$ and identify an element of $\partial \Gamma$ with the edge incident to it. 

Then it follows from \cite{bhmv} that $C_{\gamma,p}(x)$ is the cardinality of the set 
$$(\inter rP) \cap \Lambda \cap\{c_i=x_i,i\in \partial \Gamma\}\subset \R^E$$ 
where we have set
$$P=\{\tau:E\to [0,1], \tau_i\le \tau_j+\tau_k,\tau_i+\tau_j+\tau_k\le 2\}$$
$$\Lambda=\{c:E\to\Z, c_i+c_j+c_k \textrm{ odd}\}$$
In these two expressions $(i,j,k)$ runs over all triples of edges in $E$ incident to a same trivalent vertex of $\Gamma$. 

This comes from the fact that $V_p(\Sigma_\gamma,x)$ has a basis obtained by coloring the edges of $\Gamma$ with $c$ for $c\in r\inter P \cap \Lambda$ satisfying $c_i=x_i$ for all $i\in \partial\Gamma$, see Theorem  4.11 in \cite{bhmv}.

Observe that for all $i\in \Z$, $\delta^{p}_i=\sum_{s\in \Z}\delta_{i+ps}$ where $\delta_i=1$ if $i=0$ and $0$ otherwise. Plugging this into Formula \eqref{formuleF} gives $F_p(l,\alpha)=\sum_{s\in \Z^k}G_p(l,\alpha+ps)$ where $G_p$ has the same definition as $F_p$ with $\delta^p$ replaced with $\delta$. We will see that this sum is actually finite so that we are reduced to study the cyclic expansion of $G_p$. 

As an example of what will follow, taking $\gamma=\emptyset$, we have the equality $Z_p(\Sigma\times S^1)=\dim V_p(\Sigma)=\card P\cap \frac{1}{r}\Lambda$. Comparing the number of integral points with the volume gives the estimate

$$Z_p(\Sigma\times S^1)=r^{\dim P}\frac{\vol P}{\covol \Lambda}+O(r^{\dim P -1})$$
Here $\dim P=3g-3$ if $g>1$ and $1$ if $g=1$, $\vol(P)$ is computed from the Lebesgue measure on $\R^{E}$ and $\covol(\Lambda)$ is the the volume of $\R^{E}/\vec{\Lambda}$ where $\vec{\Lambda}$ is the vectorial part of the affine lattice $\Lambda$. One can compute $\covol(\Lambda)=2^{2g-3}$ if $g>1$ and $1$ if $g=1$. Notice that up to a normalization factor, the leading order is the volume of the moduli space $\homo(\pi_1(\Sigma),\textrm{SU}_2)/\textrm{SU}_2$ with the volume form associated to the Atiyah-Bott-Goldman symplectic structure. 

The proof of the general case will rely on the same kind of estimations. More precisely, let $V_n$ be the affine subspace of $\R^E$ given by the equations $l_j(x)=\alpha_j$ for $j=1,\ldots,k$ and $l_0(x)=n$. Then we have $G_p(l,\alpha)=\sum_{n\in \Z} A^{2nD}g_r(n)$ where 
$$ g_r(n)=\card (\inter rP)\cap V_n\cap \Lambda.$$

We would like to apply Lemma \ref{integral} to this situation in order to estimate $g_r(n)$. The lemma applies to the polytope $rP\cap V_n$ included in the euclidean space $V_n$ provided that $rP\cap V_n$ has non-empty interior and $V_n\cap \Lambda$ is an affine lattice in $V_n$. If the first assumption does not hold, it implies that $V_n\cap \inter(rP)=\emptyset$ and hence $g_r(n)=0$.  
Next, the linear part $\vec{V}$ of $V_n$ is by hypothesis the kernel of the linear forms $l_0,\ldots,l_k$ which are integral and linearly independent. It follows that $\vec{V}\cap \vec{\Lambda}$ is a free abelian group of rank $N-k-1$ where we have set $N=\card E=\dim P$. Hence, $V_n\cap \Lambda$ is a lattice in $V_n$ if and only if it is non-empty. 

We observe that this last condition only depends on $n$ modulo 2. Indeed, if $c:E\to \Z$ satisfies the equation defining $V_n$ and $\Lambda$ modulo $2$, we can add to $c$ a function $d:\partial \Gamma\to \Z$ which satisfies $2l_0(d)=n-l_0(c)$ and $2l_j(d)=\alpha_j-l_j(c)$. Such a function exists by the assumption that $L$ is surjective. It follows that we need to specify the parity of $n$. In the sequel, we will denote by $\nu\in \Z/2\Z$ a solution of the preceding system and restrict to those $n$'s which are congruent to $\nu$. 

For $n=\nu$ mod $2$ we have 
$$g_r(n)=r^{N-k-1}\frac{\vol{P\cap \frac{1}{r}V_n}}{\covol V_n\cap\Lambda}+O(r^{N-k-2}).$$ 
We have $\covol V_n\cap \Lambda=\covol \vec{V}\cap \vec{\Lambda}$, showing that this quantity does depend on $n$.

On the other hand, the function $\mathcal V(\alpha_0,\ldots,\alpha_k)=\vol \{ \tau \in P, l_j(\tau)=\alpha_j,j=0,\ldots,k\}$ is a continuous and piecewise polynomial function with compact support . It follows that 
$$g_r(n)=\frac{r^{N-1-k}}{\covol \vec{V}\cap\vec{\Lambda}}\mathcal{V} \left(\frac{n}{r},\frac{\alpha_1}{r}+2s_1,\ldots,\frac{\alpha_k}{r}+2s_k \right)+O(r^{N-k-2})$$
 which can be written $g_r(n)=r^{N-1-k}f(\frac{n}{r})+O(r^{N-k-2})$ for a continuous and piecewise polynomial function $f$ with compact support. This finally proves the proposition.
\end{proof}

\begin{lemma}\label{integral}
Let $P$ be a polytope with non-empty interior in an Euclidean space $E$ of dimension $N$ and $\Lambda$ be an affine lattice in $E$. We define the radius of $\Lambda$ as the constant (independent of $\lambda$) $\rho(\Lambda)=\sup\{d(x,\lambda), \text{ where }x\in E\text{ satisfies }d(x,\lambda)\le d(x,\mu)\,\forall  \mu \in \Lambda\}$. 

$$\left|\card (\inter P\cap \Lambda)-\frac{\vol P}{\covol \Lambda}\right|
\le\sum_{F} \frac{b_{c_F}\vol F \,\rho(\lambda)^{c_F}}
{\covol(\Lambda)}$$
In this formula $F$ runs over the faces of $P$ of positive codimension $c_F$ and $b_n$ is the volume of the unit ball in $\R^n$. 
\end{lemma}

\begin{proof}
For any $\lambda\in \Lambda$ define its Voronoi cell 
$$V(\lambda)=\{x\in \R^N\text{ s.t. } d(x,\lambda)\le d(x,\mu)\,\forall \mu \in \Lambda\}$$
Then we have $\bigcup_{\lambda\in \Lambda}V(\lambda)=E$, $\vol V(\lambda)=\covol(\Lambda)$ and $\inter V(\lambda)\cap \inter V(\mu)=\emptyset$ for $\lambda\ne \mu$. 
Consider $Q=\bigcup\limits_{\lambda\in \inter P \cap \Lambda}V(\lambda)$ such that $\vol Q=\card (\inter P\cap \Lambda)\vol V(\lambda)$. 
Any point $x$ in the symmetric difference $P\Delta Q$ is at distance at most $\rho=\rho(\Lambda)$ from the boundary of $\partial P$ as we have $\rho=\sup\{d(x,\lambda),x\in V(\lambda)\}$. 

Hence $|\vol P-\vol Q|\le \vol P\Delta Q\le \vol N_{\rho}(\partial P)$ where we have set $N_\rho(\partial P)=\{x\in E, d(x,\partial P)\le \rho\}$. Any point in $N_\rho(\partial P)$ has its closest point in $\partial P$ belonging to some face $F$ of $P$ of positive codimension. This gives the following estimation, proving the lemma.
$$\vol N_\rho(\partial P)\le \sum_{F} \vol F\, b_{c_F} \rho^{c_F}.$$
 \end{proof}

\subsection{A geometric interpretation of the trace}\label{geom}

Consider a surface $\Sigma$ and a weighted multicurve $(\gamma,w)$. The aim of this subsection is to provide a geometric interpretation for the evaluations $\ev_{A_p^\sigma}\tr_p [\gamma,w]$.

Recall that we have set $X(\Sigma)=\homo^{\rm irr}(\pi_1(\Sigma),\rm{SU}_2)/\rm{SU}_2$, $\nu$ the volume form and $\nu_g=\int_{X(\Sigma)}\nu$.
Every component $\gamma_i$ of $\gamma$ defines a map $\theta_i: X(\Sigma)\to [0,1]$ by the formula
\begin{equation}\label{angle}
\tr \rho(\gamma_i)=2\cos \pi\theta_i ([\rho]).
\end{equation}
Suppose that $\gamma$ is a pants decomposition of $\Sigma$ and denote by $\Gamma$ the corresponding graph. Then, the functions are well-known to be action variables on $X(\Sigma)$, see \cite{jw,mn}. Precisely, these maps Poisson commute and the joint map $\Theta=(\theta_1,\ldots,\theta_k):X(\Sigma)\to \R^k$ has image the polytope $P$ defined in the proof of Proposition \ref{comptage}. The Hamiltonian flows of the maps $\theta_i$ induce an action of $\R^E$ which act transitively on the fibers of $\Theta$. Moreover, the kernel of this action is precisely $\frac{1}{2}\vec{\Lambda}$. 

We prove here the following theorem which implies Theorem \ref{witten-sigma} and extends the main result of \cite{mn}.

\begin{theorem}\label{sigma}
Let $\sigma$ be an odd integer and set $A_p=-e^{\frac{i\pi}{p}}$. Suppose that $p$ goes to infinity such that $\sigma$ and $2p$ are coprime. 
Then 
$$\lim_{p\to\infty}\ev_{A_p^\sigma}\tr_p [\gamma,w]=\frac{1}{\nu_g}\int_{X(\Sigma)}\prod_j \tr \rho(\gamma_j)^{a_j\sigma} d\nu(\rho)$$
if $\sum_j b_j$ is even and $0$ otherwise. 
\end{theorem}
Before entering into the proof, let us see how this theorem implies Theorem \ref{witten-sigma}. For any banded link $L\subset \Sigma\times S^1$, we set $\Lambda_\sigma(L)=\lim_{p\to\infty}\ev_{A_p^\sigma}\tr_p(L)$. The compatibly with Kauffman relations implies that $\Lambda_\sigma$ is a linear form on $\mathcal{K}(\Sigma\times S^1,-1)=\mathcal{K}(\Sigma\times S^1)\underset{A=-1}{\otimes}\C$. 
On the other hand, let $\gamma=\gamma_1\cup\cdots\cup \gamma_k$ be a pants decomposition of $\Sigma$ and $\mathcal{K}^{\gamma}(\Sigma\times S^1)$ be the sub-module generated by banded links living in $N$, the product of a neighborhood of $\gamma$ with the circle. 
We denote by $\tau:X(\Sigma\times S^1)\to\{\pm 1\}$ the function satisfying $\rho(t)=\tau(\rho)$Id. Any element of $\mathcal{K}^{\gamma}(\Sigma\times S^1)$ may be viewed as a function on the character variety $X(\Sigma\times S^1)$ which depends only on $\theta_1,\ldots,\theta_k$ and $\tau$. That is, we can associate to any banded link $L\subset N$ a function $F_L(\theta_1,\ldots,\theta_k,\tau)$ satisfying 

$$\prod_i -\tr\rho(L_i)=F_L(\theta_1(\rho),\ldots,\theta_k(\rho),\tau(\rho)).$$

In this setting, we define a linear form $\Lambda_\sigma':\mathcal{K}^\gamma(\Sigma\times S^1)\to \C$ by the formula $\Lambda'_\sigma(L)=\int F_L(\sigma\theta_1,\ldots,\sigma\theta_k,\tau)d\mu$ where $\mu$ is the image of $\nu/\nu_g$ by $(\theta_1,\ldots,\theta_k,\tau)$. Hence, by construction, we have $\Lambda'_1(f)=\frac{1}{\nu_g}\int_{X(\Sigma\times S^1)}f d \nu$. 
Theorem \ref{witten-sigma} consists in proving the formula $\Lambda_\sigma(L)=\Lambda'_\sigma(L)$ for a banded link $L$ projecting without crossing.
On the other hand, weighted multicurves $[\gamma,w]$ with fixed $\gamma$ are generators for $\mathcal{K}^\gamma(\Sigma\times S^1)$, hence it is sufficient to prove the equality $\Lambda_\sigma [\gamma,w]=\Lambda_\sigma' [\gamma,w]$ for all weights $w$. 


Fix a weight $w=(a_1,b_1,\ldots,a_k,b_k)$. In the torus $\gamma_j\times [0,1]$, the element $\langle a_j,b_j\rangle_T$ corresponds to the function mapping the representation $[\rho]$ to $\tr \rho({\gamma_j}^{a_j}t^{b_j})$. Hence $F=\prod_j 2\cos(\pi a_j\theta_j)\tau^{b_j}$ and $\Lambda'_\sigma [\gamma,w]=\frac{1}{2\nu_g}\int_{X(\Sigma\times S^1)}\prod_j \tr \rho({\gamma_j}^{\sigma a_j})\tau^{\sigma b_j} d\nu(\rho)$. 

As $X(\Sigma\times S^1)$ is a disjoint union of two copies of $X(\Sigma)$ defined by the equations $t=1$ and $t=-1$, the integral vanishes if $\sum_j b_j$ is odd, and otherwise reduces to the integral in the right hand side of the equation in Theorem \ref{sigma}. The conclusion follows. 

\begin{proof}[Proof of Theorem \ref{sigma}]
With the notation of the proof of Proposition \ref{existence}, we have the formula 
$$\ev_{A_p^\sigma}\tr_p [\gamma,w]=\sum_{m,n,\epsilon,\zeta}e^{\frac{2i\pi\sigma}{p}\sum_j \epsilon_ja_j n_j}\prod_j \delta^p_{n_j+\epsilon_jb_j-\zeta_jm_j}\frac{C_{\gamma,p}(m,n)}{Z_p(\Sigma\times S^1)}.$$ 

For $C_{\gamma,p}(m,n)$ to be non-zero we need to have $0<m_j,n_j<r$. At the same time there exists $k_j\in \Z$ such that $n_j-\zeta_jm_j+\epsilon_j b_j=2rk_j$. Hence if $\zeta_j=1$, we have necessarily $k_j=0$. If $\zeta_j=-1$, we have either $k_j=0$ and $n_j+m_j=-\epsilon_j b_j$ or $k_j=1$ and $n_j+m_j=2r-\epsilon_j b_j$. The number of solutions is bounded with respect to $r$, and such terms can be neglected in the sum. Hence, we can suppose that $\zeta_j=1$ and replace $\delta^p$ with $\delta$. This gives 
$$\ev_{A_p^\sigma}\tr_p [\gamma,w]=\sum_{n,\epsilon}e^{\frac{2i\pi\sigma}{p}\sum_j \epsilon_ja_j n_j}\prod_j \frac{C_{\gamma,p}(n+\epsilon b,n)}{Z_p(\Sigma\times S^1)}+O\Big(\frac{1}{p}\Big)$$ 
where $n+\epsilon b$ is a shorthand for the tuple $(n_j+\epsilon_jb_j)$.

Let $\Gamma$ be the graph adapted to $\gamma$ introduced in the proof of Proposition \ref{comptage}. We also introduce the polytope $P$ and the lattice $\Lambda$. 

Applying Lemma \ref{integral}, we get that $C_{\gamma,p}(m,n)$ is either zero or equivalent to $\frac{\vol(rP\cap V_{m,n})}{\covol(\Lambda\cap V_{m,n})}$ where $V_{m,n}$ is the affine space given by the equations $c_{p_i}=m_i$ and $c_{q_j}=n_j$ for $j=1,\ldots,k$. Supposing non-triviality, we get further 
$$C_{\gamma,p}(m,n)=r^{\card E-2k}\frac{\vol(P\cap V_{\frac{m}{r},\frac{n}{r}})}{\covol(\Lambda\cap V_{m,n})}+O(r^{\card E-2k-1}).$$

Denote by $\hat{\Gamma},\hat{P},\hat{\Lambda}$ the data corresponding to gluing the vertices $p_j$ and $q_j$ of $\Gamma$ for $j=1,\ldots,k$. Then, the same reasoning gives $Z_p(\Sigma\times S^1)=r^{\card E-k}\frac{\vol\hat P}{\covol{\hat{\Lambda}}}+O(r^{\card E-k-1})$.

As $\frac{m}{r}$ and $\frac{n}{r}$ are equal up to $O(\frac{1}{r})$, the volume $\vol (P\cap V_{\frac{m}{r},\frac{n}{r}})$ is equal to $\vol \hat{P}\cap \hat{V}_{\frac{n}{r}}$ up to $O(\frac{1}{r})$ where $\hat{V}_{\frac n r}$ is the subspace of $\R^{\hat{E}}$ defined by the equation $c_j=\frac{n_j}{r}$. 

Let us deal with the covolumes by considering first the case of $\hat{\Gamma}$. Let $C_*(\hat{\Gamma},\Z/2\Z)$ be the cellular complex associated to $\hat{\Gamma}$. A coloring $c: \hat{E}\to \Z$ reduces modulo 2 to an element of $C_1(\hat{\Gamma},\Z/2\Z)$. In this setting, the condition defining the vectorial part of $\hat{\Lambda}$ is simply to be a cycle. It follows that the index of $\hat{\Lambda}$ in $\Z^{\hat{E}}$ is the cardinality of $C_1(\hat{\Gamma},\Z/2\Z)/\ker \partial$ which is $2^{\card \hat{E}-\dim H_1(\hat{\Gamma},\Z/2\Z)}$. 

The same reasoning applies to $\Gamma$, supposing that $\Lambda\cap V_{m,n}$ is non-empty. The vectorial part of $\Lambda \cap V_{m,n}$ is a map $c:E\to\Z$ which vanishes at boundary points and such that $c(e_i)+c(e_j)+c(e_k)$ is even for all triples $(i,j,k)$ incident to a same vertex. Its reduction modulo 2 is a cycle in $C_1(\Gamma,\Z/2\Z)$, hence, the index of $\Lambda \cap V_{m,n}$ is $2^{\card{E}-2k-\dim H_1(\Gamma,\Z/2\Z)}$.
Supposing it is non-zero, we get the following estimate where $h=-\dim H_1(\hat{\Gamma},\Z/2\Z)+\dim H_1(\Gamma,\Z/2\Z)$. 
\begin{equation}\label{truc}\frac{C_{\gamma,p}(m,n)}{Z_p(\Sigma\times S^1)}=r^{-k}\frac{\vol P\cap \hat{V}_{\frac n r}}{\vol \hat P}2^{k+h}+O(r^{-k-1}).
\end{equation}

Let $l$ be the number of connected components of $\Gamma$. The exact sequence of the pair $(\hat{\Gamma},\Gamma)$ gives the formula $h+k-l+1=0$ which replaces the power of $2$ in Equation \eqref{truc} with $2^{l-1}$. 

Consider now the problem of the non-vanishing of $C_{\gamma,p}(m,n)$. An element $c\in \Lambda \cap V_{m,n}$ satisfies $c(p_i)=m_j$, $c(q_i)=n_i$ and $c(e_i)+c(e_j)+c(e_k)=1$ for all triple of edges $e_i,e_j,e_k$ incident to a same vertex. Summing $c+1$ twice over each edge of any component $\Gamma_i$ of $\Gamma$ gives modulo 2 the equality $l_i(c)=\card \partial \Gamma_i$ where $l_i(c)=\sum_{x\in \partial \Gamma_i} c(x)$. Summing over the connected components and observing that $m_j+n_j=b_j$ modulo 2, this gives the identity $\sum_j b_j=0$. 

Conversely, a simple argument involving the homology of $(\Gamma_i,\partial \Gamma_i)$ modulo 2 implies that $C_{\gamma,p}(m,n)$ is non-zero if and only if $l_i(m,n)=0$ for all $i$.  Hence, assuming that $\sum_j b_j$ is even, we get 
$$\ev_{A_p^\sigma}\tr_p [\gamma,w]=\frac{1}{r^k}\sum_{\substack{n,\epsilon \\ l_i(n+\epsilon b,n)=0}}e^{\frac{2i\pi\sigma}{p}\sum_j \epsilon_ja_j n_j} 2^{l-1}\frac{\vol P\cap \hat{V}_{\frac n r}}{\vol \hat P}+O(\frac{1}{p})$$
The parity conditions $l_i$ should divide the sum by $2^l$ but the sum $\sum_{i=1}^l l_i(n+\epsilon b,n)=\sum b_i$ vanishes modulo $2$. Hence they divide the sum by $2^{l-1}$ and this factor cancels. 

We recognize a Riemann sum: setting $f(x)=\frac{\vol \hat P\cap \hat{V}_x}{\vol \hat P}$ we get 
\begin{eqnarray*}\lim_{p\to \infty}\ev_{A_p}\tr_p [\gamma,w]&=&\int_\R \sum_{\epsilon}e^{2i\pi\sigma\sum_j \epsilon_ja_j x_j}f(x)dx\\
&=&\int_\R\prod_j 2\cos(2\pi\sigma a_jx_j)f(x)dx.
\end{eqnarray*}

This proves the proposition as $f(x)dx$ is equal to the push-forward $\Theta_*\frac{\nu}{\nu_g}$ and by construction $\rho(\gamma_j)$ is conjugate to $\begin{pmatrix} e^{i\pi\theta_j} & 0 \\ 0 & e^{-i\pi \theta_j}\end{pmatrix}$. 
\end{proof}

\section*{Acknowledgements} 
We would like to thank Laurent Charles, Gregor Masbaum and Maxime Wolff for valuable discussions and encouragement.

\end{document}